\documentclass[a4paper,12pt]{amsart}

\usepackage[english]{babel}
\usepackage[utf8]{inputenc} 
\usepackage[T1]{fontenc}    
\usepackage[left=3cm, right=3cm, top=2cm]{geometry} 
\usepackage{indentfirst} 

\usepackage{amssymb}
\usepackage{mathtools}      
\usepackage{bbm}            
\usepackage{stmaryrd}       
\usepackage{mathrsfs}       
\usepackage{mathabx}        
\usepackage[makeroom]{cancel} 

\usepackage{palatino}       
\usepackage{microtype}      

\usepackage{graphicx}       
\usepackage{tikz}
\usepackage{tikz-cd}        
\usetikzlibrary{cd}

\usepackage{enumitem}       

\usepackage[dvipsnames]{xcolor} 
\usepackage{cite}           
\usepackage{comment}
\usepackage[normalem]{ulem} 

\theoremstyle{plain}
\newtheorem{theorem}{Theorem}[section]
\newtheorem{lemma}[theorem]{Lemma}
\newtheorem*{lemma*}{Lemma}
\newtheorem{proposition}[theorem]{Proposition}

\newtheorem{corollary}[theorem]{Corollary}

\theoremstyle{definition}
\newtheorem{definition}[theorem]{Definition}
\newtheorem{definition-notation}[theorem]{Definition/Notation}
\newtheorem{lemma-definition}[theorem]{Lemma/Definition}

\theoremstyle{remark}
\newtheorem{remark}[theorem]{Remark}

\renewcommand{\phi}{\varphi} 
\newcommand{\eps}{\varepsilon}

\newcommand{\C}{\mathbb{C}}
\newcommand{\R}{\mathbb{R}}
\newcommand{\Z}{\mathbb{Z}}

\newcommand{\K}{\mathbb{K}}

\DeclarePairedDelimiter{\lr}{(}{)}

\DeclarePairedDelimiter{\lrs}{[}{]}

\newcommand{\virg}[1]{``#1''} 

\definecolor{green}{rgb}{0,0.4,0}

\definecolor{lightgrey}{rgb}{0.7, 0.7, 0.7}

\definecolor{ballblue}{rgb}{0.0, 0.5, 1.0}

\newcommand{\comm}[1]{}

\newcommand{\cat}[1]{\mathbf{#1}}
\newcommand{\fun}[1]{\mathbf{#1}}

\DeclareMathOperator{\colim}{colim}
\newcommand{\homset}[3][]{\mathrm{Hom}_{#1}\lr{#2,#3}}

\DeclareMathOperator{\tr}{tr}
\DeclareMathOperator{\str}{str}
\newcommand{\gr}{\fun{gr}}
\newcommand{\grc}[1]{\widehat{\gr{#1}}}

\newcommand{\cl}[2][]{Cl_{#1}(#2)}

\newcommand{\rees}[1]{\mathcal{R}#1}

\newcommand{\crees}[1]{C^\omega_{\mathcal{R}}#1}
\newcommand{\creestr}[1]{C^\omega_{\mathcal{R},\tr}#1} 
\newcommand{\crki}{\crees{\K_I}}

\newcommand{\awt}[2][P]{\lr*{#2,#2_{\tr},#1,\tr_#2}} 

\newcommand{\strawt}[2][P] {\lr*{\crees{#2},\crees{#2_{\tr}},\crees{#1},\tr_{\crees{#2}}}} 

\newcommand{\commu}[1]{\{#1\}} 

\usepackage[hidelinks,draft=false]{hyperref}
\hypersetup{
    linktoc=all,    
}

\numberwithin{equation}{section}
\title{An Abstract Index Theorem via Rees Algebras}

\author{Eugenio Landi}
\address{Sapienza Universit\`a di Roma; Dipartimento di Matematica ``Guido Castelnuovo'', P.le Aldo Moro, 5 - 00185 - Roma, Italy;
}
\email{eugenio.landi@uniroma1.it}


\date{\today}

\begin{document}

\begin{abstract}
We develop a purely algebraic framework for index-type theorems based on the
Rees construction for filtered differential graded algebras (FDGAs).
Alongside the classical Rees module we introduce a \emph{smooth}
variant $\crees{A}$, adapted to analytic arguments, and we study traces and
their pointwise and coefficient-wise extensions to Rees algebras. The main result
(Corollary~\ref{thm.index}) is an abstract index theorem: given a compatible
datum of filtered differential graded associative algebras with traces $A$, $B$, $G$ with a morphism $\phi\colon A\to B$, an action
$\rho\colon G\otimes A\to A$ and a morphism $i\colon G\to B$, together with a
graded-central element, one has
\[
\tr_B\!\left(e^{f_0+f_1}\right)=\tr_{\grc A}\!\left(e^{\gamma}h\right),
\]
where the left-hand side is the trace in $B$ and the right-hand side the trace in
the Laurent series associated graded of $A$. 
Here $f_0+f_1$ is a curvature-type element, i.e., an element of the form $d_B\beta+\beta^2$ for some odd-degree element $\beta$, while $\gamma$ and $h$ are certain elements in $\grc A$. The formalism is modelled on the Getzler
rescaling technique and on the derivation of the localization formula for the loop space Chern character
by Ludewig and Yi.
\end{abstract}
\maketitle
\tableofcontents

\section{Introduction}\label{sec:intro}

Many index theorems share a common structure. One has an ``analytic'' quantity,
typically the trace of an exponential (e.g., a heat-kernel supertrace, or a Chern
character), and one wishes to show that it equals a
``topological'' or ``geometric'' quantity. In passing from the analytic side to the topological side, the relevant algebra (typically an algebra of functions with values in bounded operators on a (super-)Hilbert space, endowed with a natural filtration) degenerates into its associated graded algebra, that happens to be graded commutative.  
Moreover, the passage from the analytic to the
geometric side is almost always effected by a rescaling: an auxiliary
parameter $t$ is introduced, the analytic quantity is shown to be independent
of $t$, and the geometric quantity is recovered as the limit $t\to 0$. This is
precisely the mechanism of Getzler rescaling in the heat-kernel proof of the
Atiyah--Singer index theorem~\cite{Getzler1986,BGV}, and, more recently, of
the localization formula for the loop space Chern character of spin manifolds
of Ludewig and Yi~\cite{Ludewig_2023}.

The purpose of this paper is to isolate the algebra behind this mechanism. The
bookkeeping device that turns a filtration together with its associated graded
into a single object over the affine line is the \emph{Rees construction}: a
filtered object $(A,F)$ is replaced by the $\K[t]$-module
$\rees A=\bigoplus_a t^aF^aA$, whose fibre over $t=s\neq 0$ is $A$ and whose
fibre over $t=0$ is $\gr_F A$. Rescaling by $t$ is then literally the geometry
of this family, and the ``$t\to 0$ limit'' is the specialization to the central
fibre.

For the analytic arguments we need to differentiate and integrate along the
parameter, and to make sense of exponentials whose power series need not
converge in the na\"ive sense. We therefore introduce in
Section~\ref{sec:filtered-complexes} a smooth variant $\crees A$ of the Rees
module, consisting of germs of smooth $A$-valued functions on $\R_{>0}$
admitting a Laurent expansion compatible with the filtration. We then set up
traces and their two natural extensions to $\crees A$: a pointwise one,
defined pointwise in $t$, and a \emph{coefficient-wise} one, defined on Laurent
coefficients.

Section~\ref{sec:index} contains the abstract index theorem. In its simplest
form a single morphism $\phi\colon A\to B$ of
FDGAs with trace already produces an identity relating a trace in $B$ to a
trace on $\gr^0 A$. The full statement incorporates
an action of a third algebra $G$, whose r\^ole is to produce, in the limit, the
curvature-type factor $e^\gamma$; the combinatorial identity organizing the
perturbative expansion of $e^{f_0+f_1}$ is the Dyson series over simplices,
which we treat in Appendix~\ref{sec.appendix}. The crucial assumption one requires is that the element $e^{f_0+f_1}$ has constant trace. 
This may appear quite an ad hoc assumption, but there is a simple sufficient criterion for this condition to hold
(Proposition~\ref{prop:constant-trace-criterion}). Correspondingly, we give a form of the index theorem
(Corollary~\ref{thm.index-beta}) whose hypotheses involve only an odd element $\beta(t)$; the element $f_0+f_1$ whose exponential has constant trace is then simply the curvature $d_B\beta+\beta^2$ of $\beta$. We then record how the theory transports to the shifted setting, which is the one that one typically encounters in geometric applications, essentially because integration of forms on an $n$-dimensional oriented compact manifold is a linear operator of degree $-n$, and Section~\ref{sec:geometry} explains, following
\cite{Ludewig_2023,Getzler1986,BGV}, how the Clifford-order filtration on
integral kernels is an instance of the general framework presented in the main body of this article.

\medskip
\noindent\textbf{Acknowledgements.}
The author thanks Domenico Fiorenza for the helpful discussions and many valuable suggestions.

This work was supported by the \virg{National Group for Algebraic and Geometric Structures, and their Applications} (GNSAGA - INdAM).

\medskip
\noindent\textbf{Conventions.}
Throughout, $\K$ is a commutative $\R$-algebra that we take as ground ring. Complexes are cochain complexes, with differentials of degree $+1$;
filtrations are increasing and indexed by $\Z$. Unadorned tensor products are
over $\K$. We write $\cat{DGM}_\K$ for the category of differential graded
$\K$-modules, $\cat{DGM}^{\Z}_\K$ for that of its $\Z$-graded objects,
$\cat{DBiGM}_R$ for that of differential bigraded $R$-modules, and
$\cat{Vect}_\K$ for that of $\K$-modules. Commutators of homogeneous elements are graded:
$\commu{x,y}=xy-(-1)^{|x||y|}yx$.

\section{Filtered Complexes}\label{sec:filtered-complexes}

This section collects the algebraic and analytic constructions that will be
used throughout the paper. We begin with filtered complexes and their shifted
variants, then discuss the classical and smooth Rees constructions. Finally, we
introduce filtered differential graded algebras, modules, and traces.

\subsection{Filtered complexes and shifted filtrations}

\begin{definition}
A \emph{filtered complex} over $\K$ is a cochain complex $(C^\bullet,d)$ together with an increasing filtration by subcomplexes
\[
\cdots \subseteq F^{a-1}C \subseteq F^aC \subseteq F^{a+1}C \subseteq \cdots \subseteq C.
\]
We write such an object as $(C,F)$, or simply as $C$ when the filtration is understood. The filtration is called:
\begin{itemize}
    \item \emph{positive} if $F^aC=0$ for all $a<0$;
    \item \emph{exhaustive} if $\colim_a F^aC=C$;
    \item \emph{Hausdorff} if $\lim_a F^aC=0$.
\end{itemize}
Morphisms of filtered complexes are filtration-preserving cochain maps. We denote by $\cat{FDGM}_\K$ the resulting category.
\end{definition}

In the applications of this paper all filtrations are exhaustive, while Hausdorffness is not needed in general.

The forgetful functor
\[
U\colon \cat{FDGM}_\K\longrightarrow \cat{DGM}_\K
\]
has a right adjoint $\fun{Const}$, which endows a complex $C$ with the constant filtration
\[
F^a\fun{Const}(C)=C \qquad \text{for every } a\in\Z.
\]
This filtration is exhaustive, and not Hausdorff unless $C=0$.

For an exhaustive filtration, every element of $C$ belongs to $F^nC$ for some $n\in \mathbb{Z}$. One then has a well defined \emph{order} function
\[
\mathrm{ord}\colon C\to \Z\cup\{-\infty\}
\]
given by $\mathrm{ord}(c)=\inf\{n\in \mathbb{Z}\,|\, c\in F^nC\}$.

\begin{definition-notation}\label{def-not:shift}
Given a filtered complex $(C,F)$ and an integer $n$, the \emph{$n$-shifted filtration} on $C$ is the filtration $F[n]$ defined by
\[
F[n]^aC\coloneqq F^{a+n}C.
\]
Thus shifting the filtration by $n$ lowers the order of every element by $n$.
\end{definition-notation}

\begin{definition}
The \emph{associated graded} of a filtered complex $(C,F)$ is the graded complex
\[
\gr_F C\coloneqq \bigoplus_{a\in\Z}\gr_F^a C,
\qquad
\gr_F^a C\coloneqq F^aC/F^{a-1}C.
\]
When no confusion is possible we simply write $\gr C$.
\end{definition}

\begin{definition}\label{def:FDGM-monoidal}
The category $\cat{FDGM}_\K$ is a closed symmetric monoidal category. The tensor product is given by
\[
(C,F)\otimes (D,G)\coloneqq (C\otimes D,F\otimes G),
\]
where
\[
(F\otimes G)^a(C\otimes D)\coloneqq \sum_{i\in\Z}F^iC\otimes G^{a-i}D\subseteq C\otimes D.
\]
The internal hom is
\[
[(C,F),(D,G)]\coloneqq ([C,D],[F,G]),
\]
with filtration
\[
[F,G]^a[C,D]\coloneqq \homset[\cat{FDGM}_\K]{(C,F)}{(D,G[a])}.
\]
The monoidal unit is the ground ring endowed with the positive filtration
\[
I^a\K=
\begin{cases}
0,& a<0,\\
\K,& a\geq 0,
\end{cases}
\]
and will be denoted by $\K_I$.
\end{definition}

For a detailed discussion of this monoidal structure we refer to \cite{brotherston2024monoidal}. 
By construction, the monoidal category $\cat{FDGM}_\K$ is a monoidal subcategory of the symmetric monoidal category of cochain complexes, and so it is in particular symmetric monoidal. 
The associated graded functor $\gr\colon\cat{FDGM}_\K\to\cat{DGM}_\K^{\Z}$ is (strong) symmetric monoidal; in particular $\gr$ sends monoids to monoids.

\subsection{Classical and smooth Rees constructions}

\begin{definition}\label{def:rees}
Let $(C,F)$ be a filtered complex. Its \emph{Rees module} is the bigraded complex
\[
\rees{(C,F)}\coloneqq \bigoplus_{a\in\Z} t^aF^aC \subseteq C[t,t^{-1}].
\]
The cohomological degree is the one coming from $C$, while the filtration degree is recorded by the power of $t$. When the filtration $F$ is clear from the context, we will simply write $\rees{C}$ for $\rees{(C,F)}$.
\end{definition}

\begin{proposition}\label{prop.monoidality-rees}
The Rees construction defines a monoidal functor
\[
\rees{}\colon \cat{FDGM}_\K\longrightarrow \cat{DBiGM}_{\K[t]},
\]
and for every $s\in\K$ there are evaluation functors
\[
\bigr\vert_s\colon \cat{DBiGM}_{\K[t]}\longrightarrow
\begin{cases}
\cat{FDGM}_\K,& s\neq 0,\\
\cat{DBiGM}_\K,& s=0,
\end{cases}
\]
induced by base change along $\K[t]\to \K$, $t\mapsto s$.
Moreover, if $(C,F)$ is a filtered complex, then
\[
\rees{C}\bigr\vert_s\cong C \quad (s\neq 0),
\qquad
\rees{C}\bigr\vert_0\cong \gr_F C.
\]
\end{proposition}
\begin{proof}
Monoidality follows directly from the definition of the tensor product filtration. The specialization at $s\neq 0$ identifies $t^a x$ with $s^a x$, hence recovers $C$ with its underlying cochain complex (using exhaustiveness for the identification of the union of the $F^aC$ with $C$). For $s=0$, 
one has
\begin{align*}
\rees{C}\otimes_{\K[t]}\K&=\rees{C}/t\rees{C}
=\bigoplus_{a\in \Z}t^a(F^aC/F^{a-1}C)
=\bigoplus_{a\in \Z}t^a\gr_F^a C
=\gr_FC,
\end{align*}
where $\K$ is seen as a $\K[t]$-module via $t\mapsto 0$.
\end{proof}

The classical Rees module keeps track of the filtration algebraically. For the analytic arguments in the later sections we also need a smooth variant.

\subsection{$C^\omega$ Rees module}

\begin{remark}\label{rem:analytic-assumptions}
Since $\K$ is an $\R$-algebra, every $\K$-module $V$ is a $\R$-vector space. 
Whenever we discuss smooth Rees modules, we will assume the real vector space $V$ is a separated locally convex space, so that the spaces of smooth maps $C^\infty(\R_{>0},V)$ can be defined in the usual sense (see \cite{schwartz}).
\end{remark}

\begin{definition}\label{def:smooth-rees}
Let $(C,F)$ be a filtered complex. We denote by $C^\omega(C)$ the subcomplex of $C^\infty(\R_{>0},C)$ consisting of those maps $a(t)$ for which there exist $\varepsilon>0$, an integer $N$, and coefficients $a_n\in C$ such that
\[
a(t)=\sum_{n\geq N} a_n t^n
\qquad \text{for } 0<t<\varepsilon.
\]
To avoid cumbersome notation, we will often write  $a(t)=\sum_{n\geq N} a_n t^n$ for an element in $C^\omega(C)$, leaving it understood that the identity only holds in $0<t<\varepsilon$, for some $\varepsilon>0$.
The \emph{smooth Rees module} of $(C,F)$ is the subcomplex
\[
\crees{C}\coloneqq \left\{a(t)=\sum_{n\geq N} a_nt^n\in C^\omega(C)\,\middle|\, a_n\in F^nC \text{ for every } n\right\}.
\]
\end{definition}
 
\begin{proposition}\label{prop.smooth-monoidal}
The assignment $(C,F)\mapsto \crees{(C,F)}$ is a lax monoidal functor. The natural morphism
\[
\crees{(C,F)}\otimes_\K \crees{(D,G)}\longrightarrow \crees{\lr{C\otimes D,F\otimes G}}
\]
is given by pointwise tensor product, and $\crees{\K_I}$ is the commutative algebra of smooth $\K$-valued functions on $(0,\infty)$ admitting a Taylor expansion in a right neighborhood of $0$. In particular, for an element $a(t)=\sum_{n\geq0}a_nt^n$ of $\crees{\K_I}$ the limit $\lim_{t\to 0^+}a(t)$ exists and equals the constant Taylor coefficient, $\lim_{t\to 0^+}a(t)=a_0$.
\end{proposition}
\begin{proof}
If
\[
a(t)=\sum_{i\geq N} a_it^i \in \crees{(C,F)},
\qquad
b(t)=\sum_{j\geq P} b_jt^j \in \crees{(D,G)},
\]
then
\[
a(t)\otimes b(t)=\sum_{n\geq N+P}\left(\sum_{i+j=n} a_i\otimes b_j\right)t^n.
\]
The inner sums are finite. Since $a_i\in F^iC$ and $b_j\in G^jD$, each coefficient belongs to $(F\otimes G)^n(C\otimes D)$. The unit map $\K\to\crees{\K_I}$ sends a scalar to the corresponding constant function. Naturality is straightforward. The statements concerning  $\crees{\K_I}$ are immediate.
\end{proof}
\begin{corollary}
The functor $\crees\colon\cat{FDGM}_\K\to\cat{Vect}_\K$ refines to a functor
\[
\crees\colon\cat{FDGM}_\K\to\cat{DGM}_{\crees{\K_I}}.
\]
\end{corollary}
\begin{proof}
A lax monoidal functor sends the monoidal unit to a monoid and modules over the unit to modules over that monoid. As $\K_I$ is the monoidal unit of $\cat{FDGM}_\K$ and every object is canonically a $\K_I$-module, the claim follows from Proposition~\ref{prop.smooth-monoidal}.
\end{proof}

In moving from $\rees{}$ to $\crees{}$ we have in particular moved from Laurent polynomials to Laurent series. Accordingly, we introduce a `Laurent series version' of the associated graded.
\begin{definition}\label{def:completed-graded}
Let $(A,F)$ be a filtered module. We set
\[
\grc{A}\coloneqq \bigoplus_{n<0}\gr^n_FA\ \oplus\ \prod_{n\geq 0}\gr^n_FA.
\]
and call it the \emph{Laurent series associated graded} of $A$.
\end{definition}

\subsection{Filtered differential-graded algebras}

\begin{definition}
A \emph{(non-unital) filtered differential-graded $\K$-algebra} (FDGA) is a semigroup object in $\cat{FDGM}_\K$. Equivalently, it is a filtered complex $(A,F)$ endowed with an associative (degree zero) multiplication
\[
\cdot\colon (A,F)\otimes (A,F)\longrightarrow (A,F).
\]
\end{definition}

\begin{proposition}\label{prop:rees-of-fdga}
Let $(A,F)$ be a filtered differential graded algebra. Then $\rees{A}$ is a differential bigraded algebra and $\crees{A}$ is a differential graded algebra which is a module over $\crees{\K_I}$.
 
\end{proposition}
\begin{proof}
Immediate from Proposition~\ref{prop.monoidality-rees} and Proposition~\ref{prop.smooth-monoidal}.
\end{proof}
 
\begin{definition}\label{def:shifted-fdga}
In view of geometric applications, it is useful to also consider $n$-shifted filtered differential graded algebras, for some $n\in \Z$, i.e., filtered complexes $(A,F)$ together with a (shifted) associative product
\[
\cdot\colon (A,F)\otimes (A,F)\longrightarrow (A,F[n]).
\]
Equivalently, $(A,F)$ is an $n$-shifted FDGA if and only if $(A,F[-n])$ is an ordinary FDGA.
\end{definition}
 
\begin{remark}\label{rem:shifted-rees}
Applying the Rees construction to the
 FDGA $A[-n]$ and recalling that $F[-n]^aA=F^{a-n}A$, gives, as a submodule of $A[t,t^{-1}]$,
\[
\rees{A[-n]}=\bigoplus_{a\in\Z}t^aF^{a-n}A=t^{\,n}\,\rees{A}.
\]
This is a differential bigraded $\K[t]$-algebra: for $u\in F^{a-n}A$ and
$v\in F^{b-n}A$ one has $uv\in F^{a+b-n}A$, so
$(t^au)(t^bv)=t^{a+b}(uv)$ lies in $t^{a+b}F^{a+b-n}A\subseteq\rees{A[-n]}$.

Equivalently, one may transport the product on $\rees{A[-n]}$ to a product on $\rees{A}=\bigoplus_a t^aF^aA$ via the identification $\rees{A[-n]}=t^{n}\rees{A}$; the
transported product then raises the $t$-degree by $n$,
\[
(t^a x)\cdot (t^b y)=t^{a+b+n}(xy)
\qquad (x\in F^aA,\ y\in F^bA),
\]
so that multiplication has bidegree $(0,n)$ in (cohomological, filtration)
degree. Both descriptions specialize to the usual Rees algebra in the unshifted case, i.e., when $n=0$.
\end{remark}

\begin{proposition}\label{prop.monoidality-eval}
For any $s\neq 0$ the functors $-\otimes_{\K[t]}\K\vert_s\colon\cat{DBiGM}_{\K[t]}\to\cat{FDGM}_\K$ are monoidal; the functor $-\otimes_{\K[t]}\K\vert_0\colon\cat{DBiGM}_{\K[t]}\to\cat{DBiGM}_\K$ is monoidal.
\end{proposition}
\begin{proof}
Base change along a morphism of commutative rings is symmetric monoidal. For $s\neq 0$, the specialization $t\mapsto s$ forgets the filtration grading and recovers the underlying filtered complex; for $s=0$, it kills the ideal $(t)$ and therefore returns the associated graded object.
\end{proof}

\begin{proposition}\label{prop.monoidality-eval-smooth}
For every $s>0$, evaluation at $t=s$ defines a monoidal functor
\[
\bigr\vert_s\colon \cat{DGM}_{\crees{\K_I}}\longrightarrow \cat{FDGM}_\K.
\]
Moreover, the quotient by $(t)$ defines a monoidal functor
\[
\bigr\vert_0\colon \cat{DGM}_{\crees{\K_I}}\longrightarrow \cat{DGM}_\K,
\qquad
M\longmapsto M/tM.
\]
On $\crees{A}$ these give $\crees{A}\bigr\vert_s\cong A$ for $s>0$ and $\crees{A}\bigr\vert_0\cong \grc{A}$.
\end{proposition}
\begin{proof}
For $s>0$ the functor is induced by the algebra morphism $\crees{\K_I}\to\K$, $f\mapsto f(s)$, so monoidality is immediate; on $\crees A$ it returns $a(t)\mapsto a(s)\in A$. For $s=0$, the quotient by $(t)$ remembers exactly the classes of the Laurent coefficients modulo the next lower filtration step. This is precisely the Laurent series associated graded object: negative degrees form a direct sum because only finitely many negative coefficients occur, whereas nonnegative degrees form a product because the positive tail may be infinite; that is, $\crees{A}/t\crees{A}\cong\grc{A}$.
\end{proof}
\begin{remark}
In what follows, we will need to consider $\K$ as an FDGA with two different filtrations: either as the tensor unit $\K_I$, for which $\rees{\K_I}=\K[t]$ and $\crees{\K_I}=\crki$, or with the constant filtration $\K_c\coloneqq(\K,\mathrm{const})$, for which $\rees{\K_c}=\K[t,t^{-1}]$ and $\crees{\K_c}=C^\omega(\K)$.
\end{remark}
\subsection{Traces and their Rees extensions}

\begin{definition}\label{def.fdga-with-trace}
A \emph{filtered differential graded algebra with trace} is a quadruple
\[
\awt{A}
\]
consisting of:
\begin{itemize}
    \item an FDGA $A$;
    \item a two-sided differential ideal $A_{\tr}\subseteq A$;
    \item a filtered complex $P$;
    \item a filtered, continuous cochain map $\tr_A\colon A_{\tr}\to P$ vanishing on graded commutators in $\commu{A,A_{\tr}}$.
\end{itemize}
A morphism
\[
\lr*{A,A_{\tr},P,\tr_A}\longrightarrow \lr*{B,B_{\tr},Q,\tr_B}
\]
consists of an algebra morphism $\phi\colon A\to B$ carrying $A_{\tr}$ to $B_{\tr}$ together with a filtered cochain map $\bar\phi\colon P\to Q$ such that
\[
\bar\phi(\tr_A(a))=\tr_B(\phi(a))
\qquad \text{for every } a\in A_{\tr}.
\]
\end{definition}

In this work we will focus on traces with values in $\K$ endowed with either the constant filtration $\K_c$ or the positive filtration $\K_I$. When both $P=\K_I$ and $Q=\K_c$ occur, the map $\bar\phi$ on trace targets is always taken to be the canonical map $\K_I\to\K_c$, i.e.\ the identity of $\K$ regarded as a morphism of filtered complexes; we make this standing assumption throughout Section~\ref{sec:index}.

\begin{definition}\label{def.traces-on-rees}
Let $\awt{A}$ be a filtered differential graded algebra with trace. We define the two following ideals of $\crees{A}$:
\begin{enumerate}[label=\textup{(\roman*)}]
    \item The  \emph{pointwise smooth Rees ideal}, denoted $\creestr{A}\subseteq \crees{A}$, consists of those $a(t)$ such that $a(t)\in A_{\tr}$ for every $t>0$.
    \item The \emph{coefficient-wise smooth Rees ideal} is $\crees{A_{\tr}}\subseteq \creestr{A}$.
\end{enumerate}
The trace is then defined pointwise on $\creestr{A}$ (and so on $\crees{A_{\tr}}$) as
\[
     \tr_{\crees{A}}(a)(t)\coloneqq \tr_A(a(t)),
     \]
     and takes values in $C^\infty(\R_{>0},P)$.
\end{definition}


\begin{proposition}\label{prop.they-coincide}
Let $\awt{A}$ be a filtered differential graded algebra with trace. Then
\begin{enumerate}
    \item The pointwise smooth Rees ideal $\creestr{A}$ and the coefficient-wise smooth Rees ideal $\crees{A_{\tr}}$ are, in fact, ideals;
    \item $\tr_{\crees{A}}$ vanishes on commutators and is then a trace;
    \item if $a(t)=\sum_{n\geq N}a_nt^n$ lies in $\crees{A_{\tr}}$, then $\tr_{\crees{A}}(a)$
    lies in $\crees{P}$, its Laurent expansion in a right neighbourhood of $0$ being
    \[
    \tr_A(a(t))=\sum_{n\geq N}\tr_A(a_n)\,t^n
    \qquad (0<t<\eps).
    \]
    Consequently $\tr_{\crees{A}}$ restricts to a trace on the coefficient-wise ideal,
    \[
    \strawt{A}.
    \] 
\end{enumerate}

\end{proposition}
\begin{proof}
 Statements (1) and (2) are obvious. (3) is a consequence of $\tr_{A}$ being linear, continuous and a map of filtered complexes.
\end{proof}

\begin{lemma}\label{lem.graded-trace}
Let $\awt{A}$ be a filtered differential graded algebra with trace. Then its associated graded algebra $\gr{A}$ --- and likewise its Laurent series associated graded $\grc{A}$ and its degree-zero part $\gr^0{A}$ --- is again an algebra with trace, with ideal $\gr A_{\tr}$ (resp.\ $\grc{A_{\tr}}$, $\gr^0 A_{\tr}$), target $\gr P$ (resp.\ $\grc P$, $\gr^0 P$), and trace $\gr(\tr_A)$.
\end{lemma}
\begin{proof}
Because $A_{\tr}\subseteq A$ is a filtered two-sided ideal and $\tr_A$ is a filtered morphism, both pass to the associated graded (respectively Laurent series associated graded, respectively degree-$0$) component, and $\gr A_{\tr}$ is an ideal in $\gr A$. Since $\gr$ is strong monoidal it sends products to products and hence graded commutators to graded commutators, so $\gr(\tr_A)$ still vanishes on them.
\end{proof}

We shall also need to know how a trace behaves in cohomology. There are two ways
of passing to cohomology here, and both are used below. Throughout, the cohomology of
a filtered complex carries the filtration induced by the images of the $H(F^\bullet-)$,
and a quotient of a filtered complex carries the quotient filtration; with these
conventions every map occurring below is filtered of order $0$.
 
\begin{proposition}\label{prop:trace-on-cohomology}
Let $\awt{A}$ be an FDGA with trace. Then:
\begin{enumerate}[label=\textup{(\roman*)}]
\item $H(A)$ is a filtered graded algebra, $H(A_{\tr})$ is a filtered graded
$H(A)$-bimodule, and the map $H(A_{\tr})\to H(A)$ induced by the inclusion is a
morphism of $H(A)$-bimodules whose image is a two-sided ideal of $H(A)$;
\item the assignment
\[
\tr_{H(A)}\colon H(A_{\tr})\longrightarrow H(P),
\qquad
\tr_{H(A)}\lrs*{a}\coloneqq\lrs*{\tr_A(a)}
\]
is well defined and vanishes on graded commutators, i.e.\
$\tr_{H(A)}\lr*{\commu{\lrs*{x},\lrs*{a}}}=0$ for $\lrs*{x}\in H(A)$ and
$\lrs*{a}\in H(A_{\tr})$.
\end{enumerate}
We call $\tr_{H(A)}$ the \emph{trace induced on cohomology}. If the differential of
$P$ vanishes then $H(P)=P$ and $\tr_{H(A)}\lrs*{a}=\tr_A(a)$ for every cocycle
$a\in A_{\tr}$.
\end{proposition}
\begin{proof}
Both statements are routine verifications. For \textup{(i)} one uses that $A_{\tr}$ is
a filtered differential ideal; for \textup{(ii)}, that $\tr_A$ is a filtered cochain map and that it vanishes on $\commu{A,A_{\tr}}$ already at the level of cochains.
\end{proof}
 
Note that $H(A_{\tr})\to H(A)$ need not be injective, which is why the trace above is
defined on $H(A_{\tr})$ itself and only its \emph{image} is asserted to be an ideal.
 
Dividing by commutators before passing to cohomology costs a trace nothing, since it
vanishes on them in any case, and it lets us take traces in cohomology of a larger
supply of elements: only closedness modulo commutators is then required.
 
\begin{definition}\label{def:commutator-quotient}
Let $\awt{A}$ be an FDGA with trace. Write
$\commu{A,A_{\tr}}\subseteq A_{\tr}$ for the span of the graded commutators
$\commu{x,a}$ with $x\in A$ and $a\in A_{\tr}$. The \emph{commutator quotient} of
$A_{\tr}$ is
\[
\overline{A_{\tr}}\coloneqq A_{\tr}\big/\commu{A,A_{\tr}} ,
\]
and we write $\bar a$ for the class of $a\in A_{\tr}$.
\end{definition}
 
\begin{proposition}\label{prop:universal-trace}
With the notation of Definition~\ref{def:commutator-quotient}:
\begin{enumerate}[label=\textup{(\roman*)}]
\item $\commu{A,A_{\tr}}$ is a filtered subcomplex of $A_{\tr}$, so $d_A$ descends to a
differential on $\overline{A_{\tr}}$;
\item $\tr_A$ factors uniquely through the projection $A_{\tr}\to\overline{A_{\tr}}$,
as a cochain map $\overline{\tr}_A\colon\overline{A_{\tr}}\to P$;
\item consequently $\tr_A$ induces a map on cohomology
\[
\tr_{H}\coloneqq H\lr*{\overline{\tr}_A}\colon
H\lr*{\overline{A_{\tr}}}\longrightarrow H(P),
\qquad
\tr_{H}\lrs*{\bar a}=\lrs*{\tr_A(a)} .
\]
\end{enumerate}
\end{proposition}
\begin{proof}
Part \textup{(i)} follows from the graded Leibniz rule; \textup{(ii)} from the vanishing of $\tr_A$ on
$\commu{A,A_{\tr}}$; and
\textup{(iii)} is then immediate.
\end{proof}
 
\begin{remark}\label{rem:two-traces}
Proposition~\ref{prop:trace-on-cohomology} needs its argument to be a $d_A$-cocycle,
whereas Proposition~\ref{prop:universal-trace} needs it only to be one \emph{modulo
commutators}. It is the second that Appendix~\ref{sec.appendix} requires: the elements
$e^{\alpha(t)}$ occurring there have no class in $H(A_{\tr})$, but do have one in
$H\lr*{\overline{A_{\tr}}}$.
\end{remark}

\begin{remark}\label{rem:trace-target-degree0}
When the trace target is $P=\K_I$, one has $\gr^a\K_I=0$ for $a\neq 0$ and $\gr^0\K_I=\K$, so $\grc{\K_I}=\K$ concentrated in degree $0$. Consequently the induced trace on $\grc{A}$ is a degree-$0$ map landing in $\K$, and it is computed on the degree-$0$ component:
\[
\tr_{\grc{A}}(y)=\tr_{\gr^0 A}(y_0),\qquad y_0=\text{degree-}0\text{ part of }y.
\]
\end{remark}

\begin{lemma}\label{lem:limit-trace}
Let $\awt[\K_I]{A}$ be a filtered differential graded algebra with trace, and let
\[
a(t)=\sum_{n\geq N} a_nt^n\in \crees{A_{\tr}} .
\]
Then
\[
\lim_{t\to 0^+}\tr_A(a(t))=\tr_{\gr^0 A}([a_0]).
\]
\end{lemma}
\begin{proof}
Since $a\in\crees{A_{\tr}}$ we have $a_n\in F^nA_{\tr}$ for all $n$. As $\tr_A\colon A_{\tr}\to \K_I$ is a morphism of filtered modules and $I^n\K=0$ for $n<0$, it follows that $\tr_A(a_n)=0$ for every $n<0$. By Proposition \ref{prop.they-coincide},
\[
\tr_A(a(t))=\sum_{n\geq 0}\tr_A(a_n)t^n
\qquad (0<t<\varepsilon),
\]
which is an ordinary convergent power series near $t=0$, hence continuous there, with value $\tr_A(a_0)$ at $t=0$. On the other hand, the degree-zero component of the induced trace on $\gr A$ is the map
\[
F^0A/F^{-1}A\longrightarrow \K,
\qquad
[a_0]\longmapsto \tr_A(a_0)
\]
(Lemma~\ref{lem.graded-trace}), so the two expressions agree.
\end{proof}

\section{An Abstract Index Theorem}\label{sec:index}

In this section we state and prove our main theorem. We give some variations of it, from the least to the most complex.

Let $\awt[\K_I]{A}$ and $\awt[\K_c]{B}$ be FDGAs with traces and let $\phi\colon A\to B$ be a morphism of algebras with traces between them, with $\bar\phi\colon\K_I\to\K_c$ the canonical map from Definition~\ref{def.fdga-with-trace}.

\begin{definition}\label{def:constant-trace}
    We say that an element $a(t)\in\creestr{A}$ (resp., in $\crees{A_{\tr}}$) has \emph{constant trace} if $\tr_A a(t)\in\K\subseteq C^\infty(\R_{>0},\K)$, i.e.\ if $\tr_A(a(t))$ is independent of $t$.
\end{definition}

\begin{remark}\label{rem:constant-trace-limit}
    If $a\in\crees{A_{\tr}}$ has constant trace, then combining Definition~\ref{def:constant-trace} with Lemma~\ref{lem:limit-trace} gives
    \[
    \tr_A(a(1))=\tr_{\gr^0{A}}([a_0]).
    \]
\end{remark}

\begin{remark}\label{rem:preimage-constant-trace}
Let $b\in\creestr{B}$ be an element of constant trace, and let $a\in \creestr{A}$ be such that $\phi(a)=b$. Then $a$ has constant trace. Indeed, using that $\phi$ is a morphism of algebras with traces and $\bar\phi=\mathrm{id}_\K$,
\[
\tr_{\crees{A}}(a)=\bar\phi\!\left(\tr_{\crees{A}}(a)\right)=\tr_{\crees{B}}(\phi(a))=\tr_{\crees{B}}(b)\in \K.
\]
\end{remark}

\begin{lemma}\label{lem:easy-case2}
    Let $H\in\crees{A_{\tr}}$, let $h=[H_0]\in\gr^0{A}$, and assume $\phi(H)$ has constant trace. Then
    \[
    \tr_{B}\!\left(\phi(H)\vert_1\right)=\tr_{\gr^0 A}(h).
    \]
\end{lemma}
\begin{proof}
    We compute
    \[
    \tr_{B}\!\left(\phi(H)\vert_1\right)
    =\tr_{\crees{B}}(\phi(H))(1)
    =\tr_{\crees{A}}(H)(1)
    =\lim_{t\to0^+}\tr_{\crees{A}}(H)(t)
    =\tr_{\gr^0 A}(h).
    \]
    The first equality is the functoriality of evaluation (the trace commutes with $\vert_1$, Proposition~\ref{prop.monoidality-eval-smooth}). The second is the trace-compatibility of $\phi$ (with $\bar\phi=\mathrm{id}_\K$). The third holds because $\phi(H)$ has constant trace, hence so does $H$ by Remark~\ref{rem:preimage-constant-trace}, so $\tr_{\crees A}(H)$ is independent of $t$ and its value at $1$ equals its limit at $0$. The last equality is Lemma~\ref{lem:limit-trace} together with $h=[H_0]$.
\end{proof}

\begin{lemma}\label{lem:easy-case}
    Let $H\in\crees{A_{\tr}}$ and assume that
        \begin{itemize}
    \item $\phi(H)$ has constant trace;
        \item $\phi(H)=e^F$ for some $F\in\crees{B}$.
    \end{itemize}
    Set $h=[H_0]\in \gr^0{A}$ and $f=F(1)\in B$. Then
    \[
    \tr_B(e^f)=\tr_{\gr^0 A}(h).
    \]
\end{lemma}
\begin{proof}
This is the special case of Lemma~\ref{lem:easy-case2} in which $\phi(H)=e^F$: indeed $\phi(H)\vert_1=e^F\vert_1=e^{F(1)}=e^f$.
\end{proof}

Let now $A$, $B$, $G$ be three filtered dg algebras, with $B$ endowed with the constant filtration. We assume we are given an action of $G$ on the filtered dg-module underlying $A$,
\[
\rho\colon G\otimes A\to A,
\]
and two morphisms of filtered dg-algebras
\[
\phi\colon A\to B, \qquad i\colon G\to B.
\]
The action $\rho$ and the morphisms $i$ and $\phi$ are required to be compatible, the compatibility being expressed by the commutativity of the following diagram of filtered dg-modules:
\begin{equation}\label{eq:base-compat}
\begin{tikzcd}
    G\otimes A\arrow[r,"\rho"]\arrow[d,"i\otimes\phi"']& A \arrow[d,"\phi"]\\
    B\otimes B\arrow[r,"\cdot"] & B,
\end{tikzcd}
\end{equation}
where $\cdot$ is the product of $B$. In elements, $\phi(g\cdot a)=i(g)\cdot\phi(a)$.

A direct consequence of Propositions~\ref{prop.smooth-monoidal} and \ref{prop.monoidality-eval-smooth} is that the diagram still commutes after applying $\crees$ (or $\gr$) to everything; explicitly:
\begin{equation}\label{eq.rees-is-compatible}
\begin{tikzcd}
    \crees{G}\otimes \crees{A}\arrow[r,"\rho"]\arrow[d,"{i}\otimes{\phi}"']& \crees{A} \arrow[d,"\phi"]\\
    \crees{B}\otimes \crees{B}\arrow[r,"\cdot"] & \crees{B},
\end{tikzcd}
\end{equation}
where, to lighten the notation, we have written $\rho$ for $\crees\rho$, and similarly for the other maps. We use this convention throughout.

\begin{remark}\label{rem:gr-action}
    By Proposition~\ref{prop.monoidality-eval-smooth} the action $\rho\colon G\otimes A\to A$ induces an action $\rho_0\colon \grc{G}\otimes \grc{A}\to \grc{A}$, and the diagram of $\K$-modules
    \[
\begin{tikzcd}
    \crees{G}\otimes \crees{A}\arrow[r,"\rho"]\arrow[d,"\vert_0\otimes\vert_0"']& \crees{A} \arrow[d,"\vert_0"]\\
    \grc{G}\otimes \grc{A}\arrow[r,"\rho_0"] & \grc{A}
\end{tikzcd}
\]
commutes. In other words, writing $\cdot_0$ for the multiplication of $\grc{G}$ on $\grc{A}$, for any $g\in \crees{G}$ and $a\in \crees{A}$,
\[
(g\cdot a)\bigr\vert_0=g\bigr\vert_0\cdot_0 a\bigr\vert_0.
\]
\end{remark}

We assume that $A$ and $B$ have traces $\mathrm{tr}_A\colon A\to \K_I$ and $\mathrm{tr}_B\colon B\to \K_c$ as before, and that $\phi$ is a morphism of algebras with traces (with $\bar\phi=\mathrm{id}_\K$).

\begin{definition}\label{def:graded-central}
Let $g\in \crees{G}$. We say that the action of $g$ on $\crees{A}$ is \emph{graded-central} if the induced action on $\grc{A}$ is a morphism of $\grc{A}$-bimodules; equivalently, if it is given by left multiplication by a central element $z_g\in Z(\grc{A})$.
\end{definition}

We first record a criterion guaranteeing that the constant-trace hypothesis of
the theorem is met; its proof is deferred to Appendix~\ref{sec.appendix}, where the
material on exponentials needed for it is developed.

\begin{proposition}[Constant-trace criterion]\label{prop:constant-trace-criterion}
Let $\awt{A}$ be an FDGA with trace whose target $P$ has vanishing differential --- as
is the case for $P=\K_I$ and for $P=\K_c$, the only trace targets used in this
section. Let $\beta(t)\in\crees{A}$ be of odd degree and set
\[
\alpha(t)=d_A\beta(t)+\beta(t)^2 .
\]
Assume that $\alpha(t)$ is exponentiable and that $e^{s\alpha(t)}\in\creestr{A}$ for
every $s>0$. Then $e^{\alpha(t)}$ has constant trace in the sense of
Definition~\ref{def:constant-trace}.
\end{proposition}
\begin{proof}
See Remark~\ref{rem:zero-differential}, which deduces the statement from
Corollary~\ref{cor:t-invariance}.
\end{proof}

This is how the hypothesis is verified in practice: one exhibits a single odd
element $\beta(t)$ of which $\alpha(t)$ is the curvature.

We now come to the main statement.

Assume that
\begin{align*}
H\colon (\mathbb{R}_{>0},+)&\to (\crees{A},\cdot), &
\tau&\mapsto H_\tau
\end{align*}
is a semigroup homomorphism lifting the exponential semigroup
$\tau\mapsto e^{\tau F_0}$ of $\crees{B}$; that is, $\phi(H_\tau)=e^{\tau F_0}$
for all $\tau>0$ (see Appendix~\ref{sec.appendix} for exponentials of this kind).

\begin{remark}\label{rem:injective-lift}
If $\phi\colon A\to B$ is injective, then any family $\{H_\tau\}_{\tau>0}$ in $\crees{A}$ with $\phi(H_\tau)=e^{\tau F_0}$ is automatically a semigroup homomorphism: $\phi(H_\tau H_s)=e^{\tau F_0}e^{sF_0}=e^{(\tau+s)F_0}=\phi(H_{\tau+s})$, and injectivity gives $H_\tau H_s=H_{\tau+s}$.
\end{remark}

\begin{lemma}\label{lem:full-semigroup}
    For $j=0,1$ let $g_j\in \crees{G}$, and let $F_j\in \crees{B}$ be the elements $F_j= i(g_j)$. Assume there is a semigroup homomorphism $H\colon \mathbb{R}_{>0}\to \crees{A_{\tr}}$ such that
\[
{\phi}(H_\tau)=e^{\tau F_0}\qquad (\tau>0).
\]
Assume the action of $g_1$ on $\crees{A}$ is graded-central and let $\gamma=z_{g_1}\in Z(\grc{A})$. Put $h=H_1\bigr\vert_0\in\grc A$ and $f_j=F_j\bigr\vert_1\in B$. If the element
\[
\sum_{n=0}^\infty\int_{\Delta_n}e^{\tau_0F_0}F_1e^{\tau_1F_0}\cdots F_1e^{\tau_nF_0}\,d\sigma_n
\]
of $\crees{B}$ (equal to $e^{F_0+F_1}$, cf.\ Corollary~\ref{cor:perturbative}) has constant trace, then
\[
\sum_{n=0}^\infty\int_{\Delta_n}\mathrm{tr}_{B}\!\left(e^{\tau_0f_0}f_1e^{\tau_1f_0}\cdots f_1e^{\tau_nf_0}\right)d\sigma_n=\mathrm{tr}_{\grc{A}}\!\left(e^{\gamma} h\right).
\]
Here $\Delta_n$ and $d\sigma_n$ are as in Appendix~\ref{sec.appendix}, so that $\int_{\Delta_n}d\sigma_n=1/n!$.
\end{lemma}
\begin{proof}
By definition of $f_0,f_1$ and functoriality of evaluation,
\[
\mathrm{tr}_{B}\!\left(e^{\tau_0f_0}f_1e^{\tau_1f_0}\cdots f_1e^{\tau_nf_0}\right)=
\mathrm{tr}_{\crees{B}}\!\left(e^{\tau_0F_0}F_1e^{\tau_1F_0}\cdots F_1e^{\tau_nF_0}\right)\bigr\vert_1 .
\]
By the commutativity of \eqref{eq.rees-is-compatible}, $\phi(g_1\cdot H_\tau)=i(g_1)\cdot \phi(H_\tau)=F_1e^{\tau F_0}$, so, since $\phi$ is an algebra homomorphism,
\begin{align*}
\mathrm{tr}_{\crees{B}}\!\left(e^{\tau_0 F_0}F_1e^{\tau_1F_0}\cdots F_1e^{\tau_nF_0}\right)\bigr\vert_1
&=\mathrm{tr}_{\crees{B}}\!\left(\phi(H_{\tau_0})\phi(g_1\cdot H_{\tau_1})\cdots \phi(g_1\cdot H_{\tau_n})\right)\bigr\vert_1\\
&=\mathrm{tr}_{\crees{B}}\!\left(\phi\!\left(H_{\tau_0}(g_1\cdot H_{\tau_1})\cdots (g_1\cdot H_{\tau_n})\right)\right)\bigr\vert_1\\
&=\mathrm{tr}_{\crees{A}}\!\left(H_{\tau_0}(g_1\cdot H_{\tau_1})\cdots (g_1\cdot H_{\tau_n})\right)\bigr\vert_1,
\end{align*}
the last step being the trace-compatibility of $\phi$.

By hypothesis the element $\sum_n\int_{\Delta_n}e^{\tau_0F_0}F_1\cdots d\sigma_n$ of $\crees{B}$ has constant trace; hence so does its $\phi$-preimage $\sum_n\int_{\Delta_n}H_{\tau_0}(g_1\cdot H_{\tau_1})\cdots(g_1\cdot H_{\tau_n})\,d\sigma_n$ in $\crees{A}$ (Remark~\ref{rem:preimage-constant-trace}). Its trace may therefore be evaluated at $t=0$ instead of $t=1$; distributing $\vert_0$ through the (convergent) sum and integrals, and using Remark~\ref{rem:gr-action}, graded-centrality, and centrality of $\gamma$:
\begin{align*}
&\sum_{n=0}^\infty\int_{\Delta_n}\mathrm{tr}_{\crees{A}}\!\left(H_{\tau_0}(g_1\cdot H_{\tau_1})\cdots (g_1\cdot H_{\tau_n})\right)\bigr\vert_1\,d\sigma_n\\
&\quad=\sum_{n=0}^\infty\int_{\Delta_n}\mathrm{tr}_{\grc{A}}\!\left(H_{\tau_0}\!\vert_0\,(g_1\!\vert_0\cdot_0 H_{\tau_1}\!\vert_0)\cdots (g_1\!\vert_0\cdot_0 H_{\tau_n}\!\vert_0)\right)d\sigma_n\\
&\quad=\sum_{n=0}^\infty\int_{\Delta_n}\mathrm{tr}_{\grc{A}}\!\left(\gamma^{\,n}\,H_{\tau_0}\!\vert_0\,H_{\tau_1}\!\vert_0\cdots H_{\tau_n}\!\vert_0\right)d\sigma_n\\
&\quad=\sum_{n=0}^\infty\int_{\Delta_n}\mathrm{tr}_{\grc{A}}\!\left(\gamma^{\,n}\,(H_{\tau_0}H_{\tau_1}\cdots H_{\tau_n})\bigr\vert_0\right)d\sigma_n\\
&\quad=\sum_{n=0}^\infty\int_{\Delta_n}\mathrm{tr}_{\grc{A}}\!\left(\gamma^{\,n}\,H_{\tau_0+\tau_1+\dots+\tau_n}\bigr\vert_0\right)d\sigma_n\\
&\quad=\sum_{n=0}^\infty\int_{\Delta_n}\mathrm{tr}_{\grc{A}}\!\left(\gamma^{\,n}\,H_{1}\bigr\vert_0\right)d\sigma_n\\
&\quad=\sum_{n=0}^\infty\left(\int_{\Delta_n}d\sigma_n\right)\mathrm{tr}_{\grc{A}}(\gamma^{\,n}h)
=\sum_{n=0}^\infty\frac{1}{n!}\,\mathrm{tr}_{\grc{A}}(\gamma^{\,n}h)
=\mathrm{tr}_{\grc{A}}\!\left(e^{\gamma}h\right).
\end{align*}
Here we used that $\vert_0$ is an algebra homomorphism, that $H$ is a semigroup homomorphism (so $H_{\tau_0}\cdots H_{\tau_n}=H_{\tau_0+\dots+\tau_n}$ and $\tau_0+\dots+\tau_n=1$ on $\Delta_n$), and that $\int_{\Delta_n}d\sigma_n=1/n!$ (Appendix~\ref{sec.appendix}).
\end{proof}

 \begin{corollary}[Abstract Index Theorem]\label{thm.index}
 Under the hypotheses of Lemma~\ref{lem:full-semigroup}, and assuming $e^{f_0+f_1}$ exists in $B$,
 \[
 \mathrm{tr}_B\!\left(e^{f_0+f_1}\right)=
 \mathrm{tr}_{\grc{A}}\!\left(e^{\gamma}h\right).
 \]
 \end{corollary}
 \begin{proof}
By the perturbative expansion (Corollary~\ref{cor:perturbative}), applied in $B$ with $f_j=F_j\vert_1$,
\[
e^{f_0+f_1}=\sum_{n=0}^\infty\int_{\Delta_n}e^{\tau_0f_0}f_1e^{\tau_1f_0}\cdots f_1e^{\tau_nf_0}\,d\sigma_n .
\]
Applying the (linear, continuous) trace $\mathrm{tr}_B$ and using Lemma~\ref{lem:full-semigroup},
\[
\mathrm{tr}_B\!\left(e^{f_0+f_1}\right)
=\sum_{n=0}^\infty\int_{\Delta_n}\mathrm{tr}_B\!\left(e^{\tau_0f_0}f_1e^{\tau_1f_0}\cdots f_1e^{\tau_nf_0}\right)d\sigma_n
=\mathrm{tr}_{\grc{A}}\!\left(e^{\gamma}h\right).
\]
 \end{proof}

\begin{corollary}[Abstract Index Theorem, curvature form]\label{thm.index-beta}
Let $g_0,g_1$, $H$, $\gamma$ and $h$ be as in Lemma~\ref{lem:full-semigroup}.
Let $\beta=\beta_0+\beta_1\in\crees{B}$ be odd with $d_B\beta_0=0$, and set
\[
F_0\coloneqq\beta_0^2,
\qquad
F_1\coloneqq d_B\beta_1+\commu{\beta_0,\beta_1}+\beta_1^2 ,
\]
so that $F_0+F_1=d_B\beta+\beta^2$. Assume
\begin{enumerate}[label=\textup{(\roman*)}]
  \item $F_0$ is exponentiable and $e^{sF_0}\in\creestr{B}$ for every $s>0$;
  \item $\beta_1$ and $d_B\beta_1$ are bounded, and $\commu{\beta_0,\beta_1}$ is
  relatively bounded with respect to $F_0$: writing $\Delta\coloneqq 1-F_0$, the operator
  $\commu{\beta_0,\beta_1}\Delta^{-1/2}$ is bounded and
  $\lVert\Delta^{1/2}e^{sF_0}\rVert=O(s^{-1/2})$ as $s\to0^+$;
  \item $F_j=i(g_j)$ for $j=0,1$.
\end{enumerate}
Put $f_j\coloneqq F_j\bigr\vert_1$.
Then $F_0+F_1$ is exponentiable with $e^{s(F_0+F_1)}\in\creestr{B}$
for every $s>0$, and
\[
\mathrm{tr}_B\!\left(e^{f_0+f_1}\right)=\mathrm{tr}_{\grc{A}}\!\left(e^{\gamma}h\right).
\]
\end{corollary}
\begin{proof}
Since $\beta_0,\beta_1$ are odd, $\beta^2=\beta_0^2+\commu{\beta_0,\beta_1}+\beta_1^2$, and
$d_B\beta_0=0$ gives $d_B\beta+\beta^2=F_0+F_1$.

By (ii) also $\beta_1^2$ is bounded, so $F_1\Delta^{-1/2}$ is bounded; hence
$\lVert F_1e^{sF_0}\rVert=O(s^{-1/2})$, which is integrable at $s=0$, so that $F_1$ is a
Miyadera perturbation of $F_0$ and Lemma~\ref{lem:DP.exp} applied to $F_0$ and $F_1$
makes $F_0+F_1$ exponentiable. Expanding $e^{s(F_0+F_1)}$ by
Corollary~\ref{cor:perturbative}, in each term
$e^{\tau_0sF_0}F_1\cdots F_1e^{\tau_nsF_0}$ one has $\sum_i\tau_i=1$, so some
$\tau_i\geq(n+1)^{-1}>0$ and that factor lies in $\creestr{B}$ by (i); as $\creestr{B}$
is a two-sided ideal (Proposition~\ref{prop.they-coincide}) the whole term does, and we
assume the series converges there.

Finally, $F_0+F_1$ is the curvature of $\beta$, so
Proposition~\ref{prop:constant-trace-criterion} applies and $e^{F_0+F_1}$ has constant
trace. Corollary~\ref{thm.index} gives the conclusion.
\end{proof}

\begin{remark}\label{rem:analytic-difficulty}
In the geometric applications the main difficulties are of an analytic nature, and they exactly correspond to checking that
assumptions (i) and (ii) are satisfied. That is, one needs trace-class estimates for the semigroup generated by $F_0$,
boundedness results for $\beta_1$ and $d_B\beta_1$, and relative boundedness for $\commu{\beta_0,\beta_1}$. In the application to the Atiyah-Singer index theorem, these analytic estimates are established in
\cite{GuneysuLudewig,HanischLudewig}, and provide the hard analytic background on which \cite{Ludewig_2023} rests.\end{remark}

\begin{remark}\label{rem:interpretation}
The left-hand side of Corollary~\ref{thm.index} is an analytic quantity, the trace in $B$ of the exponential of the ``total curvature'' $f_0+f_1$. The right-hand side is topological in nature: by Remark~\ref{rem:trace-target-degree0} it is $\tr_{\gr^0A}\bigl((e^\gamma h)_0\bigr)$, a trace computed entirely on the associated graded, where $\gamma$ plays the r\^ole of a curvature form and $h$ that of a symbol. This is the abstract shell of the Getzler-rescaling identity; see Section~\ref{sec:geometry}.
\end{remark}

\begin{remark}
When working with filtered objects it can happen that the object that appears most natural is a shifted filtered one. For filtered complexes this is inconsequential, since the shift $F\mapsto F[n]$ is a self-equivalence of $\cat{FDGM}_\K$.

For FDGAs, on the other hand, the shift does not preserve the algebra structure and therefore interacts nontrivially with $\rees$ and $\crees$ (Definition~\ref{def:shifted-fdga} and Remark~\ref{rem:shifted-rees}). Nonetheless, the categories of FDGAs and of $n$-shifted FDGAs are equivalent, via $(A,F)\mapsto (A,F[-n])$. Consequently every statement of Sections~\ref{sec:filtered-complexes}--\ref{sec:index} has a shifted counterpart, and --- more importantly for the applications --- every statement about $n$-shifted FDGAs translates back to the unshifted setting by applying the shift. In particular, the index theorem of Corollary~\ref{thm.index} holds verbatim in the shifted setting once all Rees modules, evaluations and traces are transported along the equivalence $(A,F)\leftrightarrow(A,F[-n])$; concretely, one replaces $\rees A$ by $t^{n}\rees A$ as in Remark~\ref{rem:shifted-rees}.

The geometric example of Section~\ref{sec:geometry} is shifted: there the natural filtration is the Clifford-order filtration, and the trace becomes filtration-compatible only after an appropriate shift.

\end{remark}

\section{Examples and geometric motivation}\label{sec:geometry}

We indicate here how the Getzler and Ludewig--Yi proof of the Atiyah--Singer index theorem for spin manifolds \cite{Getzler1986,BGV,Ludewig_2023} fits into the general algebraic framework of Section~\ref{sec:index}, of which it was the motivating example.

Let $M$ be a closed Riemannian spin manifold of even dimension $2\ell$, with spinor bundle $S\to M$. The kernels of operators acting on spinors are sections of $S\boxtimes S^{*}\to M\times M$. Fibrewise, $\mathrm{End}(S_m)\cong \cl{T_mM}$ is a Clifford algebra, and the \emph{Clifford filtration} filters it by the number of Clifford generators; on the top piece the Clifford supertrace $\str_{\mathrm{Cl}}$ picks out (up to a normalization constant) the coefficient of the volume element, and it \emph{vanishes on elements of Clifford order} $<2\ell$. Getzler rescaling assigns to a differential operator a \emph{Getzler order}, refining the order of the operator by its Clifford degree; see \cite[Ch.~4]{BGV}.

\begin{lemma}\label{lem:geometric}
On the sections of $S\boxtimes S^{*}$, in the notation of \cite{Ludewig_2023}, define the (reverse Getzler) filtration
\[
F^p=\Bigl\{\sigma\ :\ \operatorname{Clifford\ order}\bigl(D\sigma(-,m)\big\vert_m\bigr)\le p+q\ \text{ for every }D\text{ of Getzler order}\le q\Bigr\}.
\]
Since the identity has Getzler order $0$, every $\sigma\in F^p$ has, on the diagonal, Clifford order at most $p$. Consider the shifted filtration $F[2\ell]$, i.e.\ $F[2\ell]^p=F^{p+2\ell}$. Then for every $\sigma\in F[2\ell]^k$ with $k<0$ the restriction of $\sigma$ to the diagonal has Clifford order at most $2\ell-1$, whence
\[
\str_{\mathrm{Cl}}\,\sigma_{(m,m)}=0 \qquad (m\in M).
\]
Consequently, for each $m\in M$, the map
\[
\sigma\longmapsto \str_{\mathrm{Cl}}\,\sigma_{(m,m)}
\]
is a morphism of filtered algebras from $\bigl(\Gamma(S\boxtimes S^{*}),F[2\ell]\bigr)$ to $\C_I$, i.e.\ a trace with values in the positively filtered ground field.
\end{lemma}
\begin{proof}
That $\sigma\in F^p$ has Clifford order $\le p$ on the diagonal is the case $q=0$, $D=\mathrm{id}$ of the defining condition. Under the shift $F[2\ell]^k=F^{k+2\ell}$, an element of $F[2\ell]^k$ with $k<0$ lies in $F^{k+2\ell}$ with $k+2\ell<2\ell$, so its diagonal restriction has Clifford order $\le 2\ell-1$; as recalled above, $\str_{\mathrm{Cl}}$ annihilates such elements, giving $\str_{\mathrm{Cl}}\sigma_{(m,m)}=0$. Thus $\str_{\mathrm{Cl}}(\,\cdot\,)_{(m,m)}$ vanishes on $F[2\ell]^{-1}$, i.e.\ it is filtered of order $0$ into $\C_I$; that it vanishes on graded commutators is the Clifford-supertrace property. Multiplicativity of the filtration is exactly the statement that the Getzler order is subadditive under composition \cite[Ch.~4]{BGV}.
\end{proof}

Thus $\bigl(\Gamma(S\boxtimes S^{*}),F[2\ell]\bigr)$, with the diagonal Clifford supertrace, is an instance of a (shifted) FDGA with trace whose target is $\K_I$, exactly as required by Section~\ref{sec:index}. In the localization statement of \cite{Ludewig_2023} the r\^ole of $B$ is played by an operator algebra with constant filtration, that of $A$ by the Clifford-filtered kernel algebra above, and that of $G$ by the infinitesimal rescaling symmetry; the limit $t\to0^+$ of Corollary~\ref{thm.index} is the passage to the symbol, producing the $\hat A$-form on the associated graded. Our aim in highlighting this algebraic machinery of the proof, separating it from the analytical aspects, is to be able to apply the same construction in other settings as soon as one recognizes the theorem's \emph{dramatis person\ae}. A first application, in a future work, will be the index theorem for families.

\appendix
\section{Simplices, semigroups, exponentials and perturbations}\label{sec.appendix}

We collect here the analytic results used in the main text: the integral over the standard simplex, the perturbative (Dyson) expansion of an exponential, and the constant-trace criterion of Getzler--Chern--Simons type. Throughout, $R$ is a unital Banach algebra unless otherwise stated; in the applications $R=\crees{B}$ or $R=B$.

\subsection{The simplex and the Dirichlet integral}\label{subsec:simplex}

Let
\[
\Delta_n=\Bigl\{(\tau_0,\dots,\tau_n)\in \mathbb{R}^{n+1}\ \big|\ \tau_i\geq 0,\ \textstyle\sum_{i=0}^n \tau_i=1\Bigr\}
\]
be the standard $n$-simplex. We parametrize it by the coordinates $(\tau_1,\dots,\tau_n)$ ranging over
\[
\tilde{\Delta}_n=\Bigl\{(\tau_1,\dots,\tau_n)\in \mathbb{R}^{n}\ \big|\ \tau_i\geq 0,\ \textstyle\sum_{i=1}^n \tau_i\leq1\Bigr\},\qquad \tau_0=1-\textstyle\sum_{i=1}^n\tau_i,
\]
and we equip $\Delta_n$ with the measure
\[
d\sigma_n\coloneqq d\tau_1\,d\tau_2\cdots d\tau_n .
\]
With this (unnormalized) convention $\int_{\Delta_n}d\sigma_n=\operatorname{vol}(\tilde{\Delta}_n)=1/n!$.

\begin{lemma}[Dirichlet integral; cf.\ \cite{Petrov}]\label{lemma.petrov}
For all integers $k_0,\dots,k_n\geq 0$,
\[
\int_{\Delta_n}\tau_0^{k_0}\tau_1^{k_1}\cdots\tau_n^{k_n}\,d\sigma_n
=\frac{k_0!\,k_1!\cdots k_n!}{\bigl(n+k_0+\cdots +k_n\bigr)!}.
\]
In particular $\int_{\Delta_n}d\sigma_n=1/n!$.
\end{lemma}
\begin{proof}
We compute the integral
\[
I=\int_{(\mathbb{R}_{\geq 0})^{n+1}}\left(\prod_{i=0}^n e^{-s_i}s_i^{k_i}\right)ds_0\cdots ds_n
\]
in two ways. First, by Fubini and the definition of the factorial,
\[
I=\prod_{i=0}^n\int_0^{+\infty}e^{-s_i}s_i^{k_i}\,ds_i=\prod_{i=0}^n k_i! .
\]
Second, substitute $s_i=r\tau_i$ with $r=\sum_{i=0}^n s_i\in[0,+\infty)$ and $(\tau_0,\dots,\tau_n)\in\Delta_n$; the Jacobian gives $ds_0\cdots ds_n=r^n\,dr\,d\sigma_n$, whence
\[
I=\int_0^{+\infty}e^{-r}r^{\,n+\sum_i k_i}\,dr\cdot\int_{\Delta_n}\prod_{i=0}^n\tau_i^{k_i}\,d\sigma_n
=\bigl(n+\textstyle\sum_i k_i\bigr)!\int_{\Delta_n}\prod_{i=0}^n\tau_i^{k_i}\,d\sigma_n .
\]
Comparing the two evaluations gives the claim.
\end{proof}

\begin{lemma}[Perturbative expansion]\label{lem:pert.exp}
Let $A,B\in R$ be such that the exponential series of $A$ and of $A+B$ converge. Then
\[
e^{A+B}=\sum_{n=0}^\infty \int_{\Delta_n}e^{\tau_0A}Be^{\tau_1A}B\cdots Be^{\tau_{n-1}A}Be^{\tau_nA}\,d\sigma_n .
\]
\end{lemma}
\begin{proof}
We prove the identity
\[
e^{A+\hbar B}=\sum_{n=0}^\infty \hbar^{\,n}\int_{\Delta_n}e^{\tau_0A}Be^{\tau_1A}\cdots Be^{\tau_nA}\,d\sigma_n
\]
of formal (indeed convergent) power series in $\hbar$; the statement follows by setting $\hbar=1$. Expanding $e^{A+\hbar B}=\sum_{m\ge 0}\frac{1}{m!}(A+\hbar B)^m$ and collecting, for fixed $n$, the words with exactly $n$ occurrences of $B$,
\[
[\hbar^n]\,e^{A+\hbar B}=\sum_{k_0,\dots,k_n\geq 0}\frac{1}{(n+k_0+\cdots +k_n)!}\,A^{k_0}BA^{k_1}B\cdots A^{k_{n-1}}BA^{k_n},
\]
where $m=n+\sum_i k_i$. By Lemma~\ref{lemma.petrov},
\[
\frac{1}{(n+\sum_i k_i)!}=\frac{1}{k_0!\cdots k_n!}\int_{\Delta_n}\tau_0^{k_0}\cdots\tau_n^{k_n}\,d\sigma_n ,
\]
so that
\begin{align*}
[\hbar^n]\,e^{A+\hbar B}
=&\int_{\Delta_n}\sum_{k_0,\dots,k_n\geq 0}\frac{(\tau_0A)^{k_0}B(\tau_1A)^{k_1}B\cdots B(\tau_nA)^{k_n}}{k_0!\cdots k_n!}\,d\sigma_n\\
=&\int_{\Delta_n}e^{\tau_0A}B\cdots Be^{\tau_nA}\,d\sigma_n.
\end{align*}
\end{proof}

\subsection{Semigroups and exponentials}\label{subsec:semigroups}

\begin{definition}\label{def.semigroup}
A (one-parameter) \emph{semigroup} in a Banach algebra $R$ is a strongly continuous representation of the monoid $(\R_{\ge0},+)$, i.e.\ a strongly continuous map $H\colon (\R_{\ge0},+)\to (R,\cdot)$ with $H_{t+s}=H_t\cdot H_s$ for all $t,s\ge 0$ and $H_0=1$. When the limit
\[
\lim_{u\to0^+}\frac{H_{u}-H_0}{u}
\]
exists (in $R$), we call it the \emph{generator} of the semigroup. A strongly continuous one-parameter semigroup is uniquely determined by its generator \cite[Ch.~II]{EngelNagel}.
\end{definition}

\begin{lemma}\label{lem:exp-semigroup}
If the exponential series $\exp(uA)\coloneqq \sum_{n=0}^\infty u^nA^n/n!$ converges for all $u\ge 0$, then $u\mapsto\exp(uA)$ is a one-parameter semigroup with generator $A$.
\end{lemma}
\begin{proof}
Convergence and $\commu{A,A}=0$ give $\exp(uA)\exp(vA)=\exp((u+v)A)$; norm-continuity in $u$ and $\frac{d}{du}\exp(uA)\big|_{u=0}=A$ are immediate from the series.
\end{proof}

\begin{definition}\label{def:exponentiable}
An element $A\in R$ is \emph{exponentiable} if there is a semigroup $H$ with generator $A$. In that case we write $H_u\eqqcolon e^{uA}$; by Definition~\ref{def.semigroup} this is unambiguous.
\end{definition}

It can happen that, for given $A$ and $B$, the series for $\exp(uA)$ converges but that for $\exp(u(A+B))$ does not. When $B$ is bounded (and more generally under a Miyadera-type condition on $B$), the perturbed semigroup can still be constructed via the Dyson--Phillips expansion, as follows \cite{Rhandi,EngelNagel}.

\begin{lemma}[Dyson--Phillips]\label{lem:DP.exp}
    Let $H_u=\exp(uA)$ be a semigroup with generator $A$, and let $B$ be bounded or, more generally, a Miyadera perturbation of $A$ \cite{Rhandi}. Set $H^{(0)}_u=H_u$ and, recursively,
    \[
    H^{(n+1)}_u=\int_0^u H^{(n)}_{u-s}\,B\,H^{(0)}_s\,ds .
    \]
    Then $\sum_{n=0}^\infty H^{(n)}_u$ converges and defines a semigroup with generator $A+B$; being unique, we denote it by $\exp(u(A+B))$.
\end{lemma}

\begin{corollary}\label{cor:perturbative}
With $A,B$ as in Lemma~\ref{lem:DP.exp},
\[
\exp(A+B)=\sum_{n=0}^\infty\int_{\Delta_n}e^{\tau_0A}Be^{\tau_1A}B\cdots Be^{\tau_nA}\,d\sigma_n .
\]
\end{corollary}
\begin{proof}
Unfolding the recursion of Lemma~\ref{lem:DP.exp} at $u=1$ expresses $H^{(n)}_1$ as the integral of $e^{\tau_0A}B\cdots Be^{\tau_nA}$ over the ordered simplex $\{0\le s_1\le\dots\le s_n\le 1\}$; in the barycentric coordinates $\tau_i$ this is exactly $\int_{\Delta_n}e^{\tau_0A}B\cdots Be^{\tau_nA}\,d\sigma_n$. Summing over $n$ gives the claim. (When the series for $e^{A+B}$ converges, this also follows directly from Lemma~\ref{lem:pert.exp}.)
\end{proof}

\begin{corollary}[Wilcox/Duhamel formula]\label{cor:Wilcox}
    Let $t\mapsto A(t)$ be a smooth path in $R$ such that $e^{uA(t)}$ is exponentiable and smooth in $(u,t)$. Then
    \[
    \frac{d}{dt}\,e^{A(t)}=\int_0^1 e^{(1-u)A(t)}\,\dot A(t)\,e^{uA(t)}\,du .
    \]
\end{corollary}
\begin{proof}
Fix $t$ and abbreviate $A=A(t)$, $\dot A=\dot A(t)$. Set $B(u)=e^{-uA}\,\partial_t e^{uA}$, so $B(0)=0$. Since $A$ commutes with $e^{uA}$,
\[
B'(u)=\partial_u\!\left(e^{-uA}\partial_t e^{uA}\right)
=-A e^{-uA}\partial_t e^{uA}+e^{-uA}\partial_t\!\left(A e^{uA}\right)
=e^{-uA}\dot A\, e^{uA},
\]
the two remaining terms cancelling because $e^{-uA}A=Ae^{-uA}$. Integrating from $0$ to $1$,
\[
e^{-A}\partial_t e^{A}=B(1)=\int_0^1 e^{-uA}\dot A\, e^{uA}\,du,
\]
and multiplying on the left by $e^{A}$ gives the formula. (For the operator-semigroup setting where $e^{-uA}$ need not be available, see \cite{10.1063/1.1705306}.)
\end{proof}

\subsection{Constant trace}\label{subsec:constant-trace}

We keep the assumptions of Remark~\ref{rem:analytic-assumptions} (following \cite{schwartz}): $A$ is endowed with a topology making it a separated locally convex space, and the differential $d_A\colon A\to A$ is continuous.

\begin{corollary}\label{cor:psi}
Let $\delta$ be a derivation of $C^\infty(\R_{>0},A)$, and let $\alpha(t)$, $\beta(t)$ be an even-degree and an odd-degree element, respectively, such that
\[
\delta\alpha(t)=0,\qquad \dot{\alpha}(t)=\delta\dot{\beta}(t).
\]
Assume $e^{u\alpha(t)}$ is a semigroup with generator $\alpha(t)$, and set
\[
\Psi(t)=\int_0^1 e^{(1-u)\alpha(t)}\,\dot{\beta}(t)\,e^{u\alpha(t)}\,du .
\]
Then
\[
\frac{d}{dt}\,e^{\alpha(t)}=\delta\,\Psi(t).
\]
\end{corollary}
\begin{proof}
Since $\delta$ is a derivation with $\delta\alpha=0$ and $\alpha$ is even, $\delta(\alpha^k)=0$ for all $k$, hence $\delta(e^{u\alpha})=0$. Using Corollary~\ref{cor:Wilcox} and $\dot\alpha=\delta\dot\beta$,
\begin{align*}
\frac{d}{dt}e^{\alpha(t)}
&=\int_0^1 e^{(1-u)\alpha(t)}\dot{\alpha}(t)e^{u\alpha(t)}\,du
=\int_0^1 e^{(1-u)\alpha(t)}\,\delta\dot{\beta}(t)\,e^{u\alpha(t)}\,du\\
&=\int_0^1 \delta\!\left(e^{(1-u)\alpha(t)}\dot{\beta}(t)e^{u\alpha(t)}\right)du
=\delta\int_0^1 e^{(1-u)\alpha(t)}\dot{\beta}(t)e^{u\alpha(t)}\,du=\delta\,\Psi(t),
\end{align*}
where the third equality uses that $\delta$ is a derivation and $\delta(e^{u\alpha(t)})=0$ (so the two exponential factors pass through $\delta$), the fourth that $\delta$ commutes with the integral.
\end{proof}

\begin{lemma}\label{lemma:deltabeta}
Let $\beta(t)$ be an odd-degree element of $\crees{A}$, and let
\[
\alpha(t)=d_A\beta(t)+\beta(t)^2 .
\]
Let $\delta_\beta$ be the odd derivation of $C^\infty(\R_{>0},A)$ given by $\delta_\beta=d_A+\commu{\beta(t),-}$. Then $\alpha(t)\in\crees{A}$, $\delta_\beta$ restricts to a derivation of $\crees{A}$, and
\[
\delta_\beta\alpha(t)=0,\qquad \dot{\alpha}(t)=\delta_\beta\dot{\beta}(t).
\]
\end{lemma}
\begin{proof}
$\crees{A}$ is a differential graded subalgebra of $C^\infty(\R_{>0},A)$ (Proposition~\ref{prop:rees-of-fdga}), so $\alpha(t)\in\crees{A}$ and $\delta_\beta$ restricts to it. We compute, using $d_A^2=0$ and that $\beta$ is odd:
\begin{align*}
\delta_\beta\alpha
&=d_A\alpha+\commu{\beta,\alpha}\\
&=d_A(d_A\beta+\beta^2)+\bigl(\beta(d_A\beta+\beta^2)-(d_A\beta+\beta^2)\beta\bigr)\\
&=\bigl((d_A\beta)\beta-\beta\,d_A\beta\bigr)+\bigl(\beta\,d_A\beta-(d_A\beta)\beta\bigr)=0 ,
\end{align*}
where $\commu{\beta,\alpha}=\beta\alpha-\alpha\beta$ because $\alpha$ is even, and $\beta^3$ cancels. For the second identity,
\[
\dot\alpha=\frac{d}{dt}(d_A\beta+\beta^2)=d_A\dot\beta+\dot\beta\beta+\beta\dot\beta
=d_A\dot\beta+\commu{\beta,\dot\beta}=\delta_\beta\dot\beta,
\]
since $\commu{\beta,\dot\beta}=\beta\dot\beta+\dot\beta\beta$ ($\beta,\dot\beta$ both odd).
\end{proof}

\begin{corollary}\label{cor:curvature-derivative}
With $\beta(t)$ and $\alpha(t)$ as in Lemma~\ref{lemma:deltabeta}, $\Psi(t)=\int_0^1 e^{(1-u)\alpha(t)}\dot{\beta}(t)e^{u\alpha(t)}\,du$, one has
\[
\frac{d}{dt}\,e^{\alpha(t)}=d_A\Psi(t)+\commu{\beta(t),\Psi(t)}.
\]
\end{corollary}
\begin{proof}
Immediate from Corollary~\ref{cor:psi} applied to $\delta=\delta_\beta$ and Lemma~\ref{lemma:deltabeta}.
\end{proof}

\begin{lemma}\label{lem:closed-trace}
Let $\awt{A}$ be a filtered differential graded algebra with trace and $\strawt{A}$ the corresponding coefficient-wise Rees algebra with trace. Let $\beta(t)\in\crees{A}$ be odd and $\alpha(t)=d_A\beta(t)+\beta(t)^2$. Assume $\alpha(t)$ is exponentiable and $e^{s\alpha(t)}\in\creestr{A}$ for every $s>0$, so that $\tr(e^{\alpha(t)})$ is defined. Then $\overline{e^{\alpha(t)}}$ is a cocycle of the commutator quotient $\overline{\creestr{A}}$; in particular $d_P\tr(e^{\alpha(t)})=0$, so that $[\tr(e^{\alpha(t)})]$ is a class in $H(\crees{P})$, and
\[
\tr_{H}\lrs*{\overline{e^{\alpha(t)}}}=\lrs*{\tr(e^{\alpha(t)})}
\]
by Proposition~\ref{prop:universal-trace}.
\end{lemma}
\begin{proof}
Since $d_A^2=0$ and $\beta$ is odd,
\[
d_A\alpha=d_A(d_A\beta+\beta^2)=(d_A\beta)\beta-\beta(d_A\beta)=\commu{d_A\beta,\beta}=\commu{\alpha,\beta},
\]
using $\commu{\beta,\beta^2}=\beta^3-\beta^3=0$. As $\alpha$ is even, $\commu{-,\beta}$ is a derivation and $d_A$ a derivation, so
\[
d_A\,e^{\alpha}=\commu{e^{\alpha},\beta}\quad\text{i.e.}\quad d_A e^\alpha=-\commu{\beta,e^\alpha}.
\]
The right-hand side lies in $\commu{\crees{A},\creestr{A}}$, so $d_Ae^{\alpha}$
vanishes in $\overline{\creestr{A}}$: that is, $\overline{e^{\alpha(t)}}$ is a cocycle
there, and Proposition~\ref{prop:universal-trace} gives the last assertion.
Therefore, since $\tr$ is a cochain map vanishing on graded commutators,
\[
d_P\tr(e^{\alpha})=\tr(d_A e^{\alpha})=-\tr\bigl(\commu{\beta,e^{\alpha}}\bigr)=0. \qedhere
\]
\end{proof}

\begin{corollary}\label{cor:t-invariance}
With $\beta(t)\in\crees{A}$ and $\alpha(t)=d_A\beta(t)+\beta(t)^2$ as above
(and under the same traceability hypotheses, together with
$\Psi(t)\in\creestr{A}$), the class of $\overline{e^{\alpha(t)}}$ in
$H\lr*{\overline{\creestr{A}}}$ is independent of $t$,
\[
\frac{d}{dt}\lrs*{\overline{e^{\alpha(t)}}}=0 ,
\]
and consequently, applying the induced trace $\tr_{H}$ of
Proposition~\ref{prop:universal-trace},
\[
\frac{d}{dt}\,\tr_{H}\lrs*{\overline{e^{\alpha(t)}}}=
\frac{d}{dt}\bigl[\tr(e^{\alpha(t)})\bigr]=0
\]
in $H(\crees{P})$.
\end{corollary}

\begin{proof}
By Corollary~\ref{cor:curvature-derivative} we have
$\frac{d}{dt}e^{\alpha(t)}=d_A\Psi(t)+\commu{\beta(t),\Psi(t)}$, whose second term
vanishes in $\overline{\creestr{A}}$; hence
$\frac{d}{dt}\overline{e^{\alpha(t)}}=d_A\overline{\Psi(t)}$ is a coboundary, which is
the first assertion, and applying $\tr_{H}$ gives the second. Concretely, without
passing through the quotient: the operator $d/dt$ commutes with $d_A$, hence descends to $H(\crees{P})$. Using Corollary~\ref{cor:curvature-derivative} and that $\tr$ vanishes on commutators and commutes with $d_A$,
\begin{align*}
\frac{d}{dt}\bigl[\tr(e^{\alpha(t)})\bigr]
&=\Bigl[\tr\bigl(\tfrac{d}{dt}e^{\alpha(t)}\bigr)\Bigr]
=\Bigl[\tr\bigl(d_A\Psi(t)+\commu{\beta(t),\Psi(t)}\bigr)\Bigr]\\
&=\bigl[\tr(d_A\Psi(t))\bigr]
=\bigl[d_A\tr(\Psi(t))\bigr]=0 .\qedhere
\end{align*}
\end{proof}

\begin{remark}\label{rem:sense-of-constant}
The element $e^{\alpha(t)}$ has no class in $H\lr*{\creestr{A}}$, since it is not a
$d_A$-cocycle; it is closed only modulo commutators, which is exactly what the
commutator quotient of Definition~\ref{def:commutator-quotient} records, and since a
trace vanishes on commutators nothing is lost by passing to it. It is therefore the
second of the two cohomological traces, $\tr_{H}$ rather than $\tr_{H(A)}$, that is
relevant in this appendix; see Remark~\ref{rem:two-traces}.
\end{remark}
\begin{remark}\label{rem:zero-differential}
Suppose the differential of $P$ vanishes; this is the case for $P=\K_I$ and $P=\K_c$,
the only trace targets used in Section~\ref{sec:index}. Then the differential of
$\crees{P}$ vanishes as well, so there are no coboundaries,
$H(\crees{P})=\crees{P}$, and $\tr_{H}\lrs*{\overline{e^{\alpha(t)}}}$ is simply the
element $\tr(e^{\alpha(t)})$. In this case Corollary~\ref{cor:t-invariance} states
precisely that
\[
\frac{d}{dt}\,\tr_{\crees{A}}\!\left(e^{\alpha(t)}\right)=0 ,
\]
that is, that $e^{\alpha(t)}$ has constant trace in the sense of
Definition~\ref{def:constant-trace}. This proves
Proposition~\ref{prop:constant-trace-criterion}.
\end{remark}

\begin{remark}
Corollary~\ref{cor:t-invariance} is the abstract Chern--Simons invariance underlying the ``constant trace'' hypotheses of Section~\ref{sec:index}: it guarantees that, for curvatures of the form $\alpha=d_A\beta+\beta^2$, the cohomology class of $\tr(e^{\alpha(t)})$ does not depend on $t$, so that its value at $t=1$ may be computed at $t\to 0^+$ via Lemma~\ref{lem:limit-trace}.
\end{remark}

\bibliographystyle{alpha}
\bibliography{bibliography}

\end{document}